\documentclass[12pt, reqno]{amsart}
\usepackage{amsmath, amsthm, amscd, amsfonts, amssymb, graphicx, color}
\usepackage{mathrsfs}
\usepackage{graphicx}
\usepackage{caption}
\usepackage{caption}
\usepackage{subcaption}
\usepackage[bookmarksnumbered, colorlinks, plainpages]{hyperref}
\hypersetup{colorlinks=true,linkcolor=red, anchorcolor=green, citecolor=cyan, urlcolor=red, filecolor=magenta, pdftoolbar=true}

\newtheorem{theorem}{Theorem}[section]
\newtheorem{lemma}[theorem]{Lemma}

\newtheorem{example}[theorem]{Example}
\newtheorem{question}[theorem]{Question}

\newtheorem{note}[theorem]{Note}

\theoremstyle{remark}
\newtheorem{remark}[theorem]{\bf{Remark}}
\numberwithin{equation}{section}
\makeatletter
\def\subsection{\@startsection{subsection}{2}%
  \z@{-.8\linespacing\@plus-.2\linespacing}%
  {.5\linespacing\@plus.2\linespacing}%
  {\centering\normalfont\bfseries}}
\makeatother
\allowdisplaybreaks
\begin{document}

\title [On the shape of the numerical range of some tridiagonal matrices ]{On the shape of the numerical range of some tridiagonal matrices  }
	\author[A. Ghosh, R. Birbonshi, S. Ojha, K. Paul]{Arobinda Ghosh,  Riddhick Birbonshi, Sarita Ojha and Kallol Paul}
	
\address[Ghosh]{Department of Mathematics, Jadavpur University, Kolkata 700032, West Bengal, India}
\email{7aghosh@gmail.com}

\address[Birbonshi] {Department of Mathematics, Jadavpur University, Kolkata 700032, West Bengal, India}
\email{riddhick.math@gmail.com}

\address[Ojha]{Department of Mathematics\\
Indian Institute of Engineering Science and Technology, Shibpur\\
Howrah 711103\\
West Bengal, India}
\email{sarita.ojha89@gmail.com}
\address[Paul]{ Vice-Chancellor, Kalyani University, West Bengal 741235 and Professor (on lien), Department of Mathematics, Jadavpur University, Kolkata 700032,
West Bengal, India}
\email{kalloldada@gmail.com}



\maketitle	

\begin{abstract} 
	 In this paper, we study the elliptical shape of the numerical range of tridiagonal matrices of orders $3,4$ and $5$ with  zero entries on the main diagonal, a positive constant value $a$ along the super-diagonal, and sub-diagonal entries are 
     \begin{align*}
  \begin{cases}
        r, sr^2, r^3, sr^4,\ldots, r^{n-1} & \mbox{ if $n$ is even, and }\\
        r, sr^2, r^3, sr^4,\ldots, sr^{n-1} & \mbox{ if $n$ is odd},
    \end{cases}       
     \end{align*}   
where $n$ denotes the order of the matrix, $s>0$ and $r$ is a real number. We further investigate the flat portions on the boundary of the numerical range of these $n\times n$ matrices  in the special case $r=-1$. Some results given in \cite{chien2011numerical} are obtained as a particular case of our results.
\end{abstract}

\section{Introduction and preliminaries}

Let $M_n$ denote the set of all $n\times n$ matrices. For $A\in M_n$, the numerical range of $A$, denoted by $W(A)$, is the set
$$W(A)=\{\langle Ax,x \rangle : ||x||=1\},$$ 
where  $\langle \cdot , \cdot \rangle$ is the inner product and  $||\cdot||$  is the norm on $\mathbb{C}^n$. For $A\in M_n$, $A^*$ denotes the adjoint of $A.$ It is a classical result  that the numerical range 
$W(A) $ is a convex, and compact subset of the complex plane (see   \cite{gustafson1997numerical}) containing the spectrum $\sigma(A)$ of $A$. Consequently, the convex hull of the spectrum i.e., $\mbox{conv}\,\,\sigma(A)$ is contained in  $W(A)$. Moreover, when $A$ is normal, $W(A)=\mbox{conv}\,\, \sigma(A)$, so the boundary $\partial W(A)$ of the numerical range consists exclusively of flat portions.

A matrix $A=(a_{ij})$ is a tridiagonal matrix if its entries
$a_{ij}$ are zero for $|i-j|>1$ and
 a tridiagonal matrix $A$ is  called improper if it contains a zero pair of corresponding off-diagonal entries: $a_{i,i+1}=a_{i+1,i}=0$ for some $i$, and proper otherwise.

 In \cite{chien2011numerical}, Chien and Nakazato obtained some results on the elliptical numerical ranges of some tridiagonal matrices as below:
 \begin{eqnarray}\label{A(n,r)}
     A(n,r)=\begin{bmatrix}
         0 & 1 & 0 & 0 & \cdots & 0\\
         r & 0 & 1 & 0 & \cdots & 0\\
         0 & r^2 & 0 & 1 & \cdots & 0\\
         0 & 0 & r^3 & 0 & \cdots & 0\\
         \vdots & \vdots & \vdots & \vdots & \ddots & 1\\
         0 & 0 & 0 & \cdots & r^{n-1} & 0
     \end{bmatrix},
 \end{eqnarray} 
 where $r \in \mathbb{R}$.
Motivated by their results, we consider the tridiagonal matrix of the form
 \begin{eqnarray} \label{A(n)}
\begin{bmatrix}
      0 & a & 0 & \cdots&0\\
      p_1 & 0 & a & \cdots&0\\
      0 & p_2 & 0 &\cdots&0\\
      \vdots & \vdots & \vdots &\ddots& a\\
       0 &0 &\cdots& p_{n-1}& 0
\end{bmatrix}.
\end{eqnarray}
For $n\geq 2$, the $p_i$'s are chosen so that 
\begin{equation*}
p_{i-1}=\begin{cases}
    r^{i-1}, & i \mbox{ is even}\\
     sr^{i-1}, & i \mbox{ is odd}
\end{cases}
\end{equation*}
where $i=2,3, \ldots, n$ and $a,s>0$.\\

Thus $A(n,s,a,r)$ can take the following forms for even and odd $n$; 
\begin{eqnarray}\label{even}  
\begin{bmatrix}
      0 & a & 0 & 0 & \cdots&0&\\
      r & 0 & a & 0 &\cdots &0&\\
      0 & sr^2 & 0 & a &\cdots&0&\\
      0 & 0 & r^3 & 0 &\cdots& 0&\\
      \vdots &\vdots& \vdots& \vdots&\ddots& a &\\
       0 &0 & 0 &\cdots& r^{n-1}&0&
\end{bmatrix} \end{eqnarray}
 and
\begin{eqnarray}\label{odd2}  
\begin{bmatrix}
0 & a & 0 & 0 & \cdots&0&\\
      r & 0 & a & 0 &\cdots &0&\\
      0 & sr^2 & 0 & a &\cdots&0&\\
      0 & 0 & r^3 & 0 &\cdots& 0&\\
      \vdots &\vdots& \vdots& \vdots&\ddots& a &\\
       0 &0 & 0 &\cdots& sr^{n-1}&0&
\end{bmatrix} \end{eqnarray}
respectively.
Throughout this paper,  we focus on only matrices of the form $A(n,s,a,r)$ where  $r\in \mathbb{R}$ and $a,s>0$. 

 Clearly for $a=s=1$ in (\ref{even}) and (\ref{odd2}), the matrix $A(n,s,a,r)$ becomes $A(n,r)$ as given in \eqref{A(n,r)}. For further information on the numerical range of  tridiagonal matrices and operators, we refer to \cite{brown2004matrices,chien1996numerical, chien1998geometric, chien2015numerical}.\\
The organization of the present article is as follows:\\
In section \ref{section2}, we show that the numerical range of the matrix  $A(n,s,a,r)$  defined in \eqref{A(n)}, is an elliptical disk for $n=3$. In section \ref{section3} and \ref{section4}, We give a complete characterization of the shape of the numerical range of the matrices  $A(n,s,a,r)$  defined in \eqref{even} for $n=4$ and \eqref{odd2} for $n=5$ respectively, is an elliptical disk with origin as its center. In section \ref{section5}, we investigate the flat portions of the boundary of the numerical range for  the matrices $A(n,s,a,r)$ with $r=-1$. Some results given in \cite{chien2011numerical} are obtained as a particular case of our results.

We end this section with the following theorem, which will be used repeatedly. 
\begin{theorem}\cite[Theorem 1]{chien1998geometric}\label{largest}
 Let $A$ be an $n\times n$ matrix over $\mathbb{C}$ and $\alpha,\beta\in \mathbb{R}$. Then  the numerical range 
 $W(A)$ is an elliptical disk with origin as its center, whose horizontal and vertical semi-axes of lengths $\alpha$ and $\beta$, respectively if and only if 
\begin{eqnarray}\label{st}
    \lambda_{\max}(H_{\theta}(A))=\sqrt{\alpha^2 -(\alpha^2-\beta^2) \sin^2 \theta}
\end{eqnarray}  for $\theta \in [0,2\pi),$ where $H_\theta(A)=\frac{1}{2}(e^{i \theta}A+e^{-i \theta}A^*)$ denotes the Hermitian part of the matrix $e^{i\theta }A$, and $\lambda_{\max}(H_{\theta}(A))$ denotes the largest eigenvalue of $H_\theta(A)$.
\end{theorem}

\section{Numerical range of $A(3,s,a,r)$}\label{section2}
We start with the numerical range of tridiagonal matrices of the form \eqref{odd2} corresponding to $n=3$.
\begin{theorem}\label{n=3} Suppose $a,s>0$. Then for any $r \in \mathbb{R}$, the numerical range 
     $W(A(3,s,a,r))$ is an elliptical disk whose boundary is given by 
    \begin{eqnarray*}
        x(\theta)&=&\frac{1}{2}\sqrt{2a^2+2ar+2asr^2+r^2+s^2r^4} \cos \theta,\\
        y(\theta)&=&\frac{1}{2}\sqrt{2a^2-2ar-2asr^2+r^2+s^2r^4}\sin \theta,\,\,\,\,\,\,\,\,\mbox{ $\theta \in [0,2\pi)$}.
    \end{eqnarray*}
    Further, let $ s\geq \frac{1}{a}$, $f(r)=2s^2r^3+(1-2as)r-a$, and let $x_0$ be the unique real root  of the equation $g(r)=2s^2r^3+(1+2as)r+a=0$, which is negative.\\
    If $a\leq r_2\leq r_1$  or $r_1\leq r_2\leq z_0$, then $W(A(3,s,a,r_2))\subseteq W(A(3,s,a,r_1))$, where 
    \begin{center}
        $z_0=\begin{cases}
            \min\{x_0,\beta,\gamma\}, & \mbox{ if $f(r)=0$ has two negative roots $\beta, \gamma$}\\
            x_0, &  \mbox{ if $f(r)=0$ has no negative  root.}
        \end{cases}$
    \end{center}
\end{theorem}
\begin{proof} 
 The characteristic polynomial of $H_{\theta}\big(A(3,s,a,r)\big)$ is given by 
 \begin{eqnarray*}
\mbox{det}\big(H_{\theta}(A(3,s,a,r)) -xI\big)=x\Big(\frac{1}{4}|ae^{i \theta}+re^{-i \theta}|^2-x^2+\frac{1}{4}|ae^{i \theta}+sr^2e^{-i \theta}|^2\Big).   \label{hati} 
\end{eqnarray*}
So, the largest eigenvalue of $H_{\theta}\big(A(3,s,a,r)\big)$, i.e., $\lambda_{\max}(H_{\theta}\big(A(3,s,a,r))\big)$, is determined by the quadratic equation
\[x^2=\frac{1}{4}(p^2\cos^2\theta +q^2 \sin^2 \theta),\]
 where 
 \begin{eqnarray*}
     p^2 &=& 2a^2+2ar+2asr^2+r^2+s^2r^4=(a+r)^2+(a+sr^2)^2,\\
     \mbox{and } \ q^2 &=& 2a^2-2ar-2asr^2+r^2+s^2r^4=(a-r)^2+(a-sr^2)^2.
 \end{eqnarray*}
 Thus,   \[\lambda_{\max}(H_{\theta}\big(A(3,s,a,r))\big)=\frac{1}{2}\sqrt{p^2\cos^2\theta +q^2 \sin^2 \theta}=\frac{1}{2}\sqrt{p^2 -(p^2-q^2) \sin^2 \theta}.\] 
Hence,  by Theorem \ref{largest}, it is clear that $W(A(3,s,a,r))$ is an elliptical disk centered at the origin with horizontal and vertical axes of lengths $p$ and $q$, respectively.\\ 
Hence \begin{eqnarray*}
    x(\theta)=\frac{p}{2}\cos\theta \,\, \mbox{and}\,\,\,
    y(\theta)=\frac{q}{2}\sin \theta,  
\end{eqnarray*}
$\theta \in [0,2\pi)$.
Hence the result follows.\\
The derivative of $p^2$ and $q^2$ with respect to $r$ is \begin{eqnarray*}
   2g(r) =2(2s^2r^3+2asr+r+a)
\end{eqnarray*}
 and \begin{eqnarray*}
  2f(r) = 2(r-a)+4sr(sr^2-a),
\end{eqnarray*}
respectively. Our aim is to show that the functions $f(r)$ and $g(r)$ are simultaneously positive or simultaneously negative over the specified intervals.\\ 
For $s\geq\frac{1}{a}$, using Descartes' rule of sign, we obtain,
\begin{enumerate}
    \item the equation $g(r)=0$ has unique real root $x_0$, which is  negative, and,
    \item the equation $f(r)=0$ has either one positive root and two negative roots (say, $\beta, \gamma$) or, one positive root and two complex roots.
\end{enumerate} 
Now, consider the following cases.\\
  \textbf{Case-I:} Let $r\geq a$. In this case, the expressions $f(r)$ and $g(r)$ are both non-negative. So, $p^2$ and $q^2$ are both increasing functions of $r$.\\
  \textbf{Case-II:} Let $r\leq z_0$. Then $p^2$ is a decreasing function whenever $r\leq x_0$. Thus, $f(r)\le0$ and $g(r)\leq 0$ hold together if $r\leq z_0$. Hence the result follows. 
\end{proof}
\begin{remark}
    In particular, if we take $a=s=1$ in Theorem \ref{n=3},   then the matrix $A(3,s,a,r)$ reduces into the matrix 
 $A(3,r)$ and we obtain,
 $$x(\theta)=\frac{1}{2}\sqrt{2+2r+3r^2+r^4}\cos{\theta}$$
   and $$y(\theta)=\frac{1}{2}\sqrt{2-2r-r^2+r^4}\sin{\theta},$$ where $0\leq
     \theta < 2\pi$.

 Moreover, the equation
 $g(r)=2s^2r^3+(1+2as)r+a=0$ reduces to $2r^3+3r+1=0$ and its unique real root, which is negative, is 
$\frac{(-1+\sqrt{3})^{1/3}}{2^{2/3}}-\frac{1}{2(-1+\sqrt{3})^{1/3}}$. Likewise, $f(r)= 2s^2r^3-2asr+r-a=0$ reduces to $2r^3-r-1=0$, whose only positive root is  $1$, the remaining roots being non-real. Hence, by Theorem \ref{n=3}, we obtain $z_0=x_0$.

 Therefore, by Theorem \ref{n=3}, if $r_1\leq r_2 \leq x_0=\frac{(-1+\sqrt{3})^{1/3}}{2^{2/3}}-\frac{1}{2(-1+\sqrt{3})^{1/3}}$ or $1\leq r_2 \leq r_1$, then $W(A(3,r_2))\subseteq W(A(3,r_1))$, which is established in \cite[Theorem 2]{chien2011numerical}. 
\end{remark}

\section{Numerical range of $A(4,s,a,r)$}\label{section3}
We now discuss the shape of the numerical range associated with the tridiagonal matrices defined in \eqref{even} for $n=4$. Our aim is to determine the conditions under which the  shape of the numerical range of the matrices $A(4,s,a,r)$ becomes an elliptical disk. An essential step in establishing this is to show that the largest eigenvalue of $H_{\theta}(A)$ is of the form \eqref{st} given in Theorem \ref{largest}. For this purpose, we define the following polynomials in $t$, which will be used in the subsequent results.
  \begin{eqnarray*}
  \alpha_1(t) &=& (a-t^3)^2+(a-st^2)^2+(a-t)^2+2at(1+st+t^2),\\
 \alpha_2(t) &=& 2at(1+st+t^2),\\
\alpha_3(t) &=&  5a^4+t^{12}+(s^4-2)t^8+(8a^2s^2-4a^2+1)t^4\\&& +(2s^2+4a^2)t^6+2s^2t^{10}
   +4a^2t^2+4a^2st^5+4a^2st^3,\\
\alpha_4 (t)&=&4at\bigg(st^3+a^2+(a^2+1)t^2+3a^2st+t^8+st^7\\&&+(s^2-1)t^4+(s^2-1)t^6+s^3t^5\bigg),\\
\alpha_5 (t)&=&2a^2t^2\bigg(1+t^4+(s^2-2)t^2+2st^3+2st\bigg).
\end{eqnarray*}
Consider the following polynomial 
\begin{eqnarray}\label{Gamma1-t}
    \gamma_1(t)= 8\alpha_3(t) \alpha_5(t)-8(\alpha_5(t))^2-(\alpha_4(t))^2.
\end{eqnarray}
To prove our next result, following lemmas are important.
\begin{lemma}\label{alpha_1} For $\alpha_1 (t), \alpha_2 (t), \alpha_3 (t), \alpha_4(t)$ and $\alpha_5(t)$  the following relations are hold:
    \begin{enumerate}
         \item[(i)] $\alpha_1(t)-\alpha_2(t) \geq 0$.
        \item[(ii)] $\alpha_1(t)+\alpha_2(t) \geq 0$.
        \item[(iii)] $\alpha_3(t)-\alpha_5(t)>0$.
    \end{enumerate}
\end{lemma}
\begin{proof}
    The proof is straightforward and is therefore omitted.
\end{proof}
\begin{lemma}\label{ALPHA}
  If $r$  is a real root of the equation 
 $\gamma_1(t)=0$, then $\alpha_5(r) \geq 0$.
\end{lemma}
\begin{proof}
From \eqref{Gamma1-t}, $\gamma_1(r)=0$ implies that 
    \begin{eqnarray}\label{gamma1-r}
         (\alpha_4(r)) ^2=8\alpha_5 (r) \Big( \alpha_3 (r)-\alpha_5(r)\Big).
    \end{eqnarray} It follows from  Lemma \ref{alpha_1}(iii) that $\alpha_5(r) \geq0$.
\end{proof}
\begin{remark}\label{IFFFF}
 In case of $\gamma_1(r)=0$ and $\alpha_5(r)=0$, Equation \eqref{gamma1-r} implies that $\alpha_4(r) =0$ and from Lemma \ref{alpha_1}(iii), $\alpha_3 (r)>0$. 
\end{remark}
   Suppose $a,s>0$ and let $r$ be a real root of  
 $\gamma_1(t)=0$.  Then corresponding to this $r$,  Lemma \ref{ALPHA} ensures that $\alpha_5(r)\geq0$.
Accordingly, two possible cases arise for this real root $r$.
\begin{enumerate}
    \item[Case 1:] If $\alpha_5(r)>0$, we define $R_1(r)$ and $T_1(r)$ as follows
\begin{eqnarray}\label{R-r}
R_1(r)&=&\frac{\alpha_4(r)}{2\sqrt{2\alpha_5(r)}}+\alpha_1(r)+\alpha_2(r)+\sqrt{2\alpha_5(r)},\\
T_1(r)&=&\frac{\alpha_4(r)}{2\sqrt{2\alpha_5(r)}}+\alpha_1(r)-\alpha_2(r)-\sqrt{2\alpha_5(r)}.
\end{eqnarray}  
    \item[Case 2:] If $\alpha_5(r)=0$, then from Remark \ref{IFFFF}, it follows that $\alpha_4(r)=0$  and $\alpha_3(r)>0$. In that case,  we define $R'_1(r)$ and $T'_1(r)$ as follows
\begin{eqnarray}\label{R'-r}
   R'_1(r)&=& \alpha_1(r) + \alpha_2(r) +\sqrt{\alpha_3(r)},
   \\ T'_1(r)&=& \alpha_1(r) - \alpha_2(r) +\sqrt{\alpha_3(r)}.
\end{eqnarray}
\end{enumerate}
Throughout the paper, the notations $R_1(r), T_1(r), R'_1(r)$, and  $T'_1(r)$ follow the expressions as defined above.
\begin{lemma}\label{D}
    $R_1(r) \geq 0$, $T_1(r) \geq0$, $R'_1(r) \geq 0$ and $ T'_1(r) \geq 0$.
\end{lemma}
\begin{proof}
Since $r$ is a real root of the equation 
    $\gamma_1(t)=0$, it follows from \eqref{gamma1-r} that
    \begin{eqnarray}\label{24}    \Big(\alpha_4(r)+4\alpha_5(r)\Big)^2=8\alpha_5(r)\Big(\alpha_3(r)+\alpha_4(r)+\alpha_5(r)\Big)\geq 0.
    \end{eqnarray}
     Also, from Lemma \ref{ALPHA}, we have $\alpha_5(r) \geq 0$.
Consider the two cases as follows:\\
\textbf{Case 1:} Let $\alpha_5(r) >0$. Then from \eqref{24} we have,  
     \begin{eqnarray*}
         \alpha_3(r)+\alpha_4(r)+\alpha_5(r)\geq 0.
     \end{eqnarray*}
 Hence \eqref{24} can be rewritten as, $\Big(\frac{\alpha_4(r)+4\alpha_5(r)}{2\sqrt{2\alpha_5 (r)}}\Big)^2=\alpha_3(r)+\alpha_4(r)+\alpha_5 (r)$.
    Substituting this values in \eqref{R-r}, we obtain
    \begin{eqnarray*}
        R_1(r)=\alpha_1(r)+\alpha_2(r)\pm \sqrt{\alpha_3(r)+\alpha_4(r)+\alpha_5(r)}.
    \end{eqnarray*}
    If we take
     \begin{eqnarray*}       R_1(r)=\alpha_1(r)+\alpha_2(r)+\sqrt{\alpha_3(r)+\alpha_4(r)+\alpha_5(r)},  
     \end{eqnarray*}
      then by Lemma \ref{alpha_1}(ii), we have $R_1(r)\geq 0$. Again, if  
      \begin{eqnarray*}
          R_1(r)=\alpha_1(r)+\alpha_2(r)-\sqrt{\alpha_3(r)+\alpha_4(r)+\alpha_5(r)},
      \end{eqnarray*}
       then we get,
    \begin{eqnarray}
    \Big(\alpha_1(r)+\alpha_2(r)\Big)^2-\Big(\alpha_3(r)+\alpha_4(r)+\alpha_5(r)\Big)=4(a+r)^2(a+r^3)^2\geq 0.
    \end{eqnarray}
    Hence, $R_1(r)\geq0$.\\
In a similar manner, using the identities 
\begin{eqnarray*}
&& \Big(\alpha_4(r)-4\alpha_5(r)\Big)^2=8\alpha_5(r)\Big(\alpha_3(r)-\alpha_4(r)+\alpha_5(r)\Big)\geq 0\\
\mbox{ and } &&
    \Big(\alpha_1(r)-\alpha_2(r)\Big)^2-\Big(\alpha_3(r)-\alpha_4(r)+\alpha_5(r)\Big)=4(a-r)^2(a-r^3)^2\geq 0
\end{eqnarray*}
  and following the similar previous steps, we conclude that $T_1(r) \geq 0$.\\    
    \textbf{Case 2:} Let $\alpha_5(r) =0$. Then by  Remark \ref{IFFFF}, it follows that $\alpha_4(r)=0$  and $\alpha_3(r)>0$. So, from Lemma \ref{alpha_1}[(i)-(ii)] it is obvious that both $R'_1(r)$ and $T'_1(r)$ are non-negative.
\end{proof}

\begin{lemma}\label{RR'}
 $R'_1(r)=T'_1(r)$ if and only if $r=0$. 
\end{lemma}
\begin{proof}
For $r=0$, we have 
\begin{equation*}
    \alpha_1(r)=3a^2, \alpha_2(r)=0, \alpha_3(r)=5a^4, \alpha_4(r)=0=\alpha_5(r).  
\end{equation*}
Also, from \eqref{gamma1-r}, $\gamma_1(r)=\gamma_1(0)=0$. Then $R'_1(r)=T'_1(r)=(3+\sqrt{5})a^2$.\\
Assume that there exists $r \in \mathbb{R}\setminus \{0\}$, satisfies $\gamma_1(r)=0$, $\alpha_5(r)=0$, and $R'_1(r)=T'_1(r)$.
    Since for such nonzero $r$,  $R'_1(r)=T'_1(r)$, it follows that  $\alpha_2(r)=0$. This yields $ 2ar(1+sr+r^2)=0$. Since $r\neq 0$ and $a>0$, it follows that  
    \begin{eqnarray}\label{E2}
        1+sr+r^2=0.
    \end{eqnarray}
   Since $r$ satisfies the equation $\alpha_5(t)=0$, so we have  
   \begin{eqnarray}\label{EF1}
       1+r^4+(s^2-2)r^2+2sr^3+2sr=0.
   \end{eqnarray}
Again from Equation \eqref{E2}, we have 
\begin{eqnarray*}
        r^2=-1-sr.
    \end{eqnarray*}
Substituting the value of  $r^2$  in \eqref{EF1}, we obtain  
    \begin{eqnarray*}
       0= 1+r^4+(s^2-2)r^2+2sr^3+2sr =4(1+sr),
    \end{eqnarray*} which implies $r=-\frac{1}{s}$. 
So, from Equation \eqref{E2}, we get $\frac{1}{s^2}=0$, which is a contradiction.

 Thus there does not exist any $r \neq 0$, which satisfies  $\gamma_1(t)=0$, $\alpha_5(t)=0$,  and $R'_1(r)=T'_1(r)$.
 Hence the result follows.
\end{proof}

\begin{theorem}\label{W4}
  Suppose $a,s>0$. Then for $r\in \mathbb{R}$, the numerical range $W(A(4,s,a,r))$ forms a circular disk with the origin as its center if and only if $r=0$ and a non-degenerate elliptical disk (non circular disk) centered at origin if and only if at least one of $a,s$ and $r$ is not equal to $1$ and  $r$ is a real root of the polynomial
    \begin{eqnarray}\label{q(4x)}
        u(t)\nonumber &=& a^4+a^4st-(a^4s^2+3a^4)t^2+(a^4s-a^2s)t^3+(a^4-a^2)t^4\\&& \nonumber+(s^3a^2+2a^2s)t^5+(2a^2-4a^2s^2)t^6+(a^2s^3+s+2a^2s)t^7\\&& \nonumber-(a^2+s^2)t^8-(-s^3+s+a^2s)t^9+(s^2-s^4)t^{10}\\&&+(s^3-s)t^{11}-s^2t^{12}+st^{13}.
    \end{eqnarray}
\end{theorem}
\begin{proof}
The characteristic polynomial of
$H_{\theta}(A(4,s,a,r)$ is given by
\begin{eqnarray*}
\mbox{det}(H_{\theta}(A(4,s,a,r)) -xI)&=&x^4-\frac{1}{4}x^2\big(|ae^{i \theta}+r^3e^{-i \theta}|^2+|ae^{i \theta}+sr^2e^{-i \theta}|^2\\&&+|ae^{i \theta}+te^{-i \theta}|^2\big)
+\frac{1}{16}|ae^{i \theta}+re^{-i \theta}|^2|ae^{i \theta}+r^3e^{-i \theta}|^2\\
&=& x^4-\frac{x^2}{4}\big(((a+r^3)^2+(a+sr^2)^2+(a+r)^2)\cos^2 \theta\\&&
+((a-r^3)^2+(a-sr^2)^2+(a-r)^2)\sin^2 \theta \big)\\ && +\frac{1}{16}((a+r)^2\cos^2 \theta
+(a-r)^2\sin ^2 \theta)\\&&((a+r^3)^2\cos^2 \theta+(a-r^3)^2\sin ^2 \theta). 
\end{eqnarray*}
From this, it follows that the largest eigenvalue of $H_{\theta}\big(A(4,s,a,r)\big)$, i.e., $\lambda_{\max} (H_{\theta}\big(A(4,s,a,r))\big)$ is determined by the quadratic equation
 \[x^2=\frac{1}{2}(-P\pm \sqrt{P^2-4Q}),\] where

\begin{eqnarray*}
P&=&-\frac{1}{4}\bigg(\big((a+r^3)^2+(a+sr^2)^2+(a+r)^2\big)\cos^2 \theta\\&&
+\big((a-r^3)^2+(a-sr^2)^2+(a-r)^2 \big)\sin^2 \theta\bigg)\\ \mbox{and } \ Q&=&\frac{1}{16} \bigg((a+r)^2\cos^2 \theta
+(a-r)^2\sin ^2 \theta \bigg)\\&&\bigg((a+r^3)^2\cos^2 \theta+(a-r^3)^2\sin ^2 \theta \bigg).
\end{eqnarray*}
Hence,
\small{\begin{eqnarray}\label{A1}
&& \lambda_{\max}(H_{\theta}(A(4,s,a,r))) \nonumber\\
&& = \frac{1}{2\sqrt{2}}\sqrt{\alpha_1(r)+\alpha_2(r)\cos2\theta+\sqrt{\alpha_3(r)+\alpha_4(r)\cos2\theta+\alpha_5(r)\cos4\theta}}.
 \end{eqnarray}}
Now, from \eqref{Gamma1-t}, by  calculating we determine the expression,
 \begin{eqnarray}\label{DD}
    \nonumber\gamma_1(r)\nonumber&=&64a^2r^2\big(sr^{13}-s^2r^{12}+(s^3-s)r^{11}+(s^2-s^4)r^{10}-(-s^3+s+a^2s)r^9\\&&\nonumber-(a^2+s^2)r^8+(a^2s^3+s+2a^2s)r^7+(2a^2-4a^2s^2)r^6+(s^3a^2+2a^2s)r^5\\&&+(a^4-a^2)r^4+(a^4s-a^2s)r^3-(a^4s^2+3a^4)r^2+a^4sr+a^4\big).
 \end{eqnarray} 
 If $a=s=r=1$, the matrix  $A(4,s,a,r)$ is  normal, in particular Hermitian, hence  its numerical range $W(A(4,s,a,r))$ is a line segment.
  Suppose  that at least one of $a,s$ and $r$ is not equal to $1$ and $r$ be a  real root of the equation $u(t)=0$ as defined in \eqref{q(4x)}. Consequently, $r$ satisfies the equation $\gamma_1(t)=0$. By Lemma \ref{ALPHA}, we obtain  $\alpha_5(r)\geq0$. We now consider two cases as follows:\\ 
 \textbf{Case 1:} Let $\alpha_5(r)>0$.  Then by \eqref{A1}, 
 \begin{eqnarray*}
\lambda_{\max}(H_{\theta}(A(4,s,a,r)))&=&\frac{1}{2\sqrt{2}}\sqrt{\alpha_1(r)+\alpha_2(r)\cos2\theta+\sqrt{2\alpha_5(r)}\Big(\cos2\theta+\frac{\alpha_4(r)}{4\alpha_5(r)}\Big)}\\ &=&  \frac{1}{2\sqrt{2}}\sqrt{R_1(r)-(R_1(r)-T_1(r))\sin^2\theta},
 \end{eqnarray*}
where $R_1(r)\geq 0, \ T_1(r)\geq0$ follows from Lemma \ref{D}. Hence by Theorem \ref{largest}, the numerical range of the matrix  $A(4,s,a,r)$ is an elliptical disk with the origin as its center.\\
\textbf{Case 2:} Let $\alpha_5(r)=0$.
 Then from Remark \ref{IFFFF}, $\alpha_4(r)=0$ and $\alpha_3(r) >0$. So, from \eqref{A1},  
 \begin{eqnarray}\label{44}
\lambda_{\max}(H_{\theta}(A(4,s,a,r)))&=&\nonumber\frac{1}{2\sqrt{2}}\sqrt{\alpha_1(r)+\alpha_2(r)\cos2\theta+\sqrt{\alpha_3(r)}}\\ \nonumber &=& \frac{1}{2\sqrt{2}}\sqrt{\alpha_1(r)+\alpha_2(r)+\sqrt{\alpha_3(r)}-2\alpha_2(r) \sin^2 \theta}\\ &=& \frac{1}{2\sqrt{2}} \sqrt{R'_1(r)-(R'_1(r)-T'_1(r))\sin^2 \theta}.
\end{eqnarray}
where $R'_1(r)\geq 0, \ T'_1(r)\geq0$ follows from Lemma \ref{D}.
Hence by Theorem \ref{largest}, the numerical range of the matrix  $A(4,s,a,r)$ is an elliptical disk with the origin as its center.\\
For both the cases, if at least one of $a,s$ and $r$ is not equal to $1$, then the matrix $A(4,s,a,r)$ is not a normal matrix. Hence, the elliptical numerical range is necessarily non-degenerate.\\
When $r=0$,  the matrix  $A(4,s,a,r)$ reduces to a weighted shift matrix. Consequently, its numerical range $W(A(4,s,a,r))$ forms a circular disk with the origin as its center.\\
Conversely, if $W(A(4,s,a,r))$ is non-degenerate elliptical disk with centered at the origin whose horizontal and vertical semi-axes have lengths $p_1(r)$ and $q_1(r)$, respectively. Then at least one of $a,s$ and $r$ is not equal to $1$, otherwise the matrix becomes Hermitian. Now, by using Theorem \ref{largest}, we have 
\begin{align*}
\lambda_{\max}(H_{\theta}(A(4,s,a,r))) =\sqrt{(p_1(r))^2 -\left((p_1(r))^2-(q_1(r))^2\right) \sin^2 \theta}.
\end{align*}
Comparing with \eqref{A1}, we obtain,
$ \gamma_1(r)=0$. Hence, from Lemma \ref{ALPHA}, we have $\alpha_5(r)\geq 0$. Also, by the above discussions, we have 
\begin{equation}\label{p_1,q_1}
 8(p_1(r))^2=
         \begin{cases}
       R_1(r)  & \mbox{ if $\alpha_5(r)>0$  }\\
       R'_1(r) & \mbox{ if $\alpha_5(r)=0$}
    \end{cases}  \ \mbox{ and } \ 
    8(q_1(r))^2=
    \begin{cases}
       T_1(r)  & \mbox{ if $\alpha_5(r)>0$ }\\
       T'_1(r) & \mbox{ if $\alpha_5(r)=0$}.
    \end{cases}
     \end{equation}
Then from  \eqref{DD}, it follows that either $r=0$ or $r$ is a real root of  $u(t)=0$. 
Next, our aim is to prove that if $W(A(4,s,a,r))$ is a non-degenerate elliptical disk with  origin as its center, which is not circular, then $r$ is a real root of $u(t)=0$, whereas, if  $W(A(4,s,a,r))$ forms a circular disk  with  origin as its center, then $r=0$.\\
Now, suppose $W(A(4,s,a,r))$ is a non-degenerate elliptical disk with the origin as its center, which is not circular, i.e., $p_1(r)\neq q_1(r)$. For $r= 0$, we have $\alpha_5(r) =0$. Using \eqref{p_1,q_1}, we get $p_1(r)=\sqrt{\frac{R'_1(r)}{8}}$ and $q_1(r)=\sqrt{\frac{T'_1(r)}{8}}$. Now by Lemma \ref{RR'}, it follows that $p_1(r)=\sqrt{\frac{R'_1(r)}{8}}=\sqrt{\frac{T'_1(r)}{8}} =q_1(r)$, which is a contradiction. Therefore, $r$ is nonzero, and hence $r$ is a real root of $u(t)=0$.\\
Again, if $W(A(4,s,a,r))$ forms a circular disk with the origin as its center. Then, by Theorem \ref{largest}, we have $p_1(r)=q_1(r)$. If possible let $r \neq 0$,  and $r$ satisfies the equation $\gamma_1(t)=0$. Then by Lemma \ref{ALPHA},  we have $\alpha_5(r) \geq 0$.
Now we consider two cases as follows:\\
\textbf{Case 1:} Let $\alpha_5(r) > 0$. Then $p_1(r)=\sqrt{\frac{R_1(r)}{8}}$ and $q_1(r)=\sqrt{\frac{T_1(r)}{8}}$. 
Therefore, $R_1(r)= T_1(r)$  implies  that \begin{eqnarray*}
     \alpha_2(r)=-\sqrt{2\alpha_5(r)},
 \end{eqnarray*} 
 which yields $16a^2r^4=0$. Since $r\neq 0$, this leads to a contradiction.\\
\textbf{Case 2:} Let $\alpha_5(r) = 0$. Then $p_1(r)=\sqrt{\frac{R'_1(r)}{8}}$ and $q_1(r)=\sqrt{\frac{T'_1(r)}{8}}$. Since $R'_1(r) = T'_1(r)$, Lemma \ref{RR'} implies that $r= 0$, which contradicts the assumption.
Hence the result follows.
\end{proof}
    \begin{remark}
    In particular, by choosing $a=s=1$ 
    in Theorem \ref{W4}, we obtain
    \begin{eqnarray*}
        u(t)&=&t^{13}-t^{12}-t^9-2t^8+4t^7-2t^6+3t^5-4t^2+t+1 \\&=&
        (t-1)^4d(t),
    \end{eqnarray*}
    where $d(t)=t^9+3t^8+6t^7+10t^6+14t^5+15t^4+14t^3+10t^2+5t+1$.
 So, the equation $u(t)=0$ holds if and only if either $t=1$ or $d(t)=0$. Since $a=s=1$, and if $t=1$ then the matrix $A(4,1,1,1)$ is Hermitian and $W(A(4,1,1,1))$ reduces to a line segment. Then it follows from Theorem \ref{W4}, $W(A(4,s,a,r))$ is a non-degenerate elliptical disk (noncircular disk) with the origin as its center if and only if $r$ is a real root of $d(t)=0$, and a circular disk with the origin as its center if and only if $r=0$,
 as established in \cite[Theorem 3]{chien2011numerical}.
\end{remark}
\section{Numerical range of $A(5,s,a,r)$}\label{section4}
We next investigate the numerical range of the tridiagonal matrices defined in \eqref{odd2} corresponding to $n=5$. Our goal is to determine when the numerical range of  $A(5,s,a,r)$ is an elliptical disk. For this purpose, it is essential that the largest eigenvalue of $H_{\theta}(A)$ be of the form given in \eqref{st} of Theorem \ref{largest}. Accordingly, we introduce  the following polynomials in $t$ and establish some preparatory lemmas.
 \begin{eqnarray*}
  \beta_1(t) &=& 4a^2+s^2t^8+t^6+t^2+s^2t^4,\\
 \beta_2(t) &=& 2at(st^3+1+st+t^2),\\
\beta_3(t) &=&4a^4+(1+6a^2s^2-4a^2)t^4+(s^4-2+2a^2s^2)t^8+(1-2s^4)t^{12}\\&&+s^4t^{16}+4a^2st^3+4a^2st^7+2s^2t^{14}+2a^2t^2+(6a^2-4a^2s^2+2s^2)t^6,\\
\beta_4(t)&=&4at^2\bigg(s^3t^{10}+(s-s^3)t^{8}+s^2t^{9}+(-s+s^3)t^4+(-s^2+1)t^7\\&&+(-s^3+s)t^6+(2a^2+1)t+(-1+s^2)t^3+2a^2s+st^2+(s^2-1)t^5\bigg),\\
\beta_5(t)&=&2a^2t^2\bigg(s^2t^6+2st^5-2s^2t^4+t^4+s^2t^2-2t^2+2st+1\bigg).
\end{eqnarray*}
Consider the following polynomial
\begin{eqnarray}\label{Gamma2-t}
\gamma_2(t)=8\beta_3(t) \beta_5(t)-8(\beta_5(t))^2-(\beta_4(t)) ^2.
\end{eqnarray}
To prove our next result, the following lemmas are important.
\begin{lemma}\label{beta}
For $\beta_1(t), \beta_2(t), \beta_3(t), \beta_4(t)$  and $\beta_5(t)$ the following relations are hold:
\begin{enumerate}
    \item[(i)] $\beta_1(t) + \beta_2(t) \geq 0$.
    \item[(ii)] $\beta_1(t) - \beta_2(t) \geq 0$.
    \item[(iii)] $\beta_3(t)-\beta_5(t)>0$. 
\end{enumerate}
\end{lemma}
\begin{proof}
    The proof is straightforward and is therefore omitted.
\end{proof}
\begin{lemma}\label{BETA}
    If $r$ is a real root of the equation  
   $\gamma_2(t)=0$, then $\beta_5(r) \geq 0$.
\end{lemma}
\begin{proof}
    From \eqref{Gamma2-t}, $\gamma_2(r)=0$ implies that
    \begin{eqnarray}\label{gamma2-r}
        (\beta_4(r)) ^2=8\beta_5(r)\Big( \beta_3(r)-\beta_5(r)\Big).
    \end{eqnarray}
    It follows from Lemma \ref{beta}(iii), $\beta_5(r) \geq 0$.
\end{proof}
\begin{remark}\label{REM}
    
 In case of $\gamma_2(r)=0$ and  $\beta_5(r) = 0$, then \eqref{gamma2-r} implies that $\beta_4(r)=0$ and from Lemma \ref{beta}(iii), $\beta_3(r) >0$. 
\end{remark}

    Suppose $a,s>0$, and let $r$ be a real root of 
        $\gamma_2(t)=0$.
        By Lemma \ref{BETA}, we have $\beta_5(r)\geq 0$. Thus there are two possible cases for the real root $r$.
\begin{enumerate}
    \item[Case 1:]
 If $\beta_5(r) >0$, we define $R_2(r)$ and $T_2(r)$ as follows
\begin{eqnarray}\label{R1-r}
R_2(r)&=&\frac{\beta_4(r)}{2\sqrt{2\beta_5(r)}}+\beta_1(r)+\beta_2(r)+\sqrt{2\beta_5(r)},\\
T_2(r)&=&\frac{\beta_4(r)}{2\sqrt{2\beta_5(r)}}+\beta_1(r)-\beta_2(r)-\sqrt{2\beta_5(r)}.
\end{eqnarray}
\item[Case 2:]
  If $\beta_5(r) =0$, then from Remark \ref{REM}, it follows that $\beta_4(r)=0$ and $\beta_3(r) >0$. In that case, we define $R'_2(r)$ and $T'_2(r)$ as follows
\begin{eqnarray}\label{R1'-r}
R'_2(r)&=& \beta_1(r) + \beta_2(r) +\sqrt{\beta_3(r)},\\ T'_2(r)&=& \beta_1(r) - \beta_2(r) +\sqrt{\beta_3(r)}.
\end{eqnarray}
\end{enumerate}
Here, the notations  $R_2(r), T_2(r), R'_2(r)$ and  $T'_2(r)$ follow the expressions as defined above.

\begin{lemma}\label{R'}
    $R'_2(r) \geq 0$ and $T'_2(r) \geq 0$.
\end{lemma}
\begin{proof}
Let $\beta_5(r) =0$, then from Remark \ref{REM}, it follows that $\beta_4(r)=0$ and $\beta_3(r) >0$.
    From Lemma \ref{beta}[(i)-(ii)], it follows directly that  $R'_2(r) \geq 0$ and $T'_2(r) \geq 0$.
\end{proof}
\begin{lemma}\label{ST} 
     $R_2'(r)=T'_2(r)$ if and only if $r=0$.
\end{lemma}
\begin{proof}
For $r=0$, we have 
\begin{eqnarray*}
\beta_1(r)=4a^2, \beta_3(r)=4a^4, \beta_5(r)=\beta_4(r)=\beta_2(r)=0.
\end{eqnarray*}
Also, from \eqref{gamma2-r},
$\gamma_2(r)=\gamma_2(0)=0$. Then $R'_2(r)=6a^2=T'_2(r)$.\\
Assume that there exists $r \in \mathbb{R}\setminus \{0\}$ satisfies $\gamma_2(r)=0$, $\beta_5(r)=0$ and $R'_2(r)=T'_2(r)$.
Then we have 
 $\beta_2(r)= 0$. This yields
 $2ar(sr^3+1+sr+r^2)=0$.
 Since $r\neq 0$ and $a>0$, it follows that 
 \begin{eqnarray*}
     sr^3+1+sr+r^2=0.\\
 \mbox{i.e., } \ 
     (sr+1)(r^2+1)=0.
 \end{eqnarray*}
 From the above equation, we have $sr+1=0$. Again $r$ satisfies the equation $\beta_5(t)=0$, which implies that
\begin{eqnarray}\label{six}
    s^2r^6+2sr^5-2s^2r^4+r^4+s^2r^2-2r^2+2sr+1=0.
\end{eqnarray}
Substituting the value   $r=-\frac{1}{s}$  into the equation \eqref{six}, we get $-\frac{4}{s^2}=0$, which is a contradiction.
 
Thus there does not exist any $r \neq 0$, which satisfies  $\gamma_2(r)=0$, $\beta_5(r)=0$,  and $R'_2(r)=T'_2(r)$.
 Hence, the result follows.  
\end{proof}

\begin{theorem}\label{W5}
 Suppose $a,s>0$ and \begin{eqnarray}\label{v}
v(t)\nonumber&=&2s^5t^{19}+(s^2-3s^4)t^{18}+(2s^3-2s^5)t^{17}+(2s^4-2s^2)t^{16}\nonumber\\&&+(4s^3-2s^5)t^{15}+(-2a^2s^3-8s^3+2s^5+2s)t^{13}-2a^2s^4t^{12}\nonumber\\&&+(2a^2s+4s^3-2s)t^{11}+(2s^2+4a^2s^4-4a^2s^2-2s^4)t^{10}\nonumber\\&&+(8a^2s^3-2s-4a^2s+2s^3)t^9+(-2a^2s^4-2a^2-2s^2)t^8\nonumber\\&&+(8a^2s+2s-4a^2s^3)t^7+(2a^4s^2-4a^2s^2+4a^2)t^6+(2a^2s^3+4a^4s)t^5\nonumber\\&&-(2a^2+4a^4s^2)t^4-(2a^2s+4a^4s)t^3-4a^4t^2+4a^4st+2a^4.
   \end{eqnarray} Then for $r \in \mathbb{R}$, the numerical range $W(A(5,s,a,r))$ forms a circular disk with origin as its center if and only if $r=0$ and a non-degenerate elliptical disk (non circular disk)  centered at origin if and only if at least one of $a,s$ and $r$ is not equal to $1$, and    $r$ is a real root of the equation $v(t)=0$ satisfying either  $ R_2(r)\geq 0$,  $T_2(r)\geq 0$ or $\beta_5(r)=0$.
\end{theorem}
\begin{proof}
The characteristic polynomial of $H_{\theta}(A(5,s,a,r))$ is given by, 
\begin{eqnarray*}
\mbox{det}\left(H_{\theta}(A(5,s,a,r)) -xI\right)&=&-x\bigg(x^4-\frac{1}{4}x^2\big(|ae^{i \theta}+r^3e^{-i \theta}|^2+|ae^{i \theta}+sr^2e^{-i \theta}|^2\\&&+|ae^{i \theta}+re^{-i \theta}|^2+|ae^{i \theta}+sr^4e^{-i \theta}|^2\big)
\\&&+\frac{1}{16}\big(|ae^{i \theta}+re^{-i \theta}|^2|ae^{i \theta}+r^3e^{-i \theta}|^2\\&&+|ae^{i \theta}+re^{-i \theta}|^2|ae^{i \theta}+sr^4e^{-i \theta}|^2\\&&+|ae^{i \theta}+sr^2e^{-i \theta}|^2|ae^{i \theta}+sr^4e^{-i \theta}|^2\big)\bigg)\\ &=& -x\bigg(x^4-\frac{x^2}{4}\big(\cos^2\theta((a+sr^4)^2+(a+r^3)^2+(a+r)^2\\&&+(a+sr^2)^2\big))+\sin^2\theta((a-sr^4)^2+(a-r^3)^2+(a-r)^2\\&&+(a-sr^2)^2\big))+\frac{1}{16}\bigg((a+sr^4)^2(a+sr^2)^2\cos^4\theta\\&&+(a+sr^4)^2(a+r)^2\cos^4\theta+(a+r^3)^2(a+r)^2\cos^4\theta\\&&+(a-sr^4)^2(a-sr^2)^2\sin^4\theta+(a-sr^4)^2(a-r)^2\sin^4\theta\\&&+(a-r^3)^2(a-r)^2\sin^4\theta+\bigg((a+sr^4)^2(a-sr^2)^2\\&&+(a-sr^4)^2(a+sr^2)^2+(a+sr^4)^2(a-r)^2\\&&+(a-sr^4)^2(a+r)^2+(a+r^3)^2(a-r)^2\\&&+(a-r^3)^2(a+r)^2\bigg)\sin^2\theta \cos^2 \theta\bigg)\bigg)\\ &=&-x\bigg(x^4-\frac{1}{4}P_1x^2+\frac{1}{16}Q_1\bigg)
\end{eqnarray*}
where 
\begin{eqnarray*}
P_1&=&\cos^2\theta\bigg((a+sr^4)^2+(a+r^3)^2+(a+r)^2+(a+sr^2)^2\bigg)\\&&+\sin^2\theta\bigg((a-sr^4)^2+(a-r^3)^2+(a-r)^2+(a-sr^2)^2\bigg),
\,\,\,\mbox{and} 
    \\Q_1&=& \cos^4\theta\bigg((a+sr^4)^2(a+sr^2)^2+(a+sr^4)^2(a+r)^2+(a+r^3)^2(a+r)^2\bigg)\\&&+\sin^4\theta\bigg((a-sr^4)^2(a-sr^2)^2+(a-sr^4)^2(a-r)^2+(a-r^3)^2(a-r)^2\bigg)\\&&+\sin^2\theta \cos^2 \theta\bigg((a+sr^4)^2(a-sr^2)^2+(a-sr^4)^2(a+sr^2)^2\\&&+(a+sr^4)^2(a-r)^2+(a-sr^4)^2(a+r)^2+(a+r^3)^2(a-r)^2\\&&+(a-r^3)^2(a+r)^2\bigg).
\end{eqnarray*}
Therefore,
\small{\begin{align}\label{Lambda}
 \lambda_{\max}(H_{\theta}(A(5,s,a,r)))\nonumber &=\frac{1}{2\sqrt{2}}\sqrt{P_1+\sqrt{P_1^2-4Q_1}}\\&=\frac{1}{2\sqrt{2}}\sqrt{\beta_1(r)+\beta_2(r)\cos2\theta+\sqrt{\beta_3(r)+\beta_4(r)\cos2\theta+\beta_5(r)\cos4\theta}}.   
\end{align}} 
 Now from \eqref{Gamma2-t}, by calculating we determine the expression 
 \begin{eqnarray}
\nonumber \gamma_2(r)\nonumber
&=&32a^2r^2\bigg(2s^5r^{19}+(s^2-3s^4)r^{18}+(2s^3-2s^5)r^{17}+(2s^4-2s^2)r^{16}\\&&\nonumber+(4s^3-2s^5)r^{15}+(-2a^2s^3-8s^3+2s^5+2s)r^{13}-2a^2s^4r^{12}\\&&\nonumber+(2a^2s+4s^3-2s)r^{11}+(2s^2+4a^2s^4-4a^2s^2-2s^4)r^{10}\\&&\nonumber+(8a^2s^3-2s-4a^2s+2s^3)r^9+(-2a^2s^4-2a^2-2s^2)r^8\\&&\nonumber+(8a^2s+2s-4a^2s^3)r^7+(2a^4s^2-4a^2s^2+4a^2)r^6+(2a^2s^3+4a^4s)r^5\\&&\nonumber-(2a^2+4a^4s^2)r^4-(2a^2s+4a^4s)r^3-4a^4r^2+4a^4sr+2a^4\bigg)\\&=&\label{r19}32a^2r^2v(r),
 \end{eqnarray} where the polynomial $v(t)$ is given by \eqref{v}.
If $a=s=r=1$, the matrix  $A(5,s,a,r)$ is  normal, in particular Hermitian, hence  its numerical range $W(A(5,s,a,r))$ is a line segment.
  Now, suppose  that at least one of $a,s$ and $r$ is not equal to $1$ and $r$ is a  real root of the equation $v(t)=0$ satisfying either $R_2(r) \geq 0$, $T_2(r)\geq 0$ or $\beta_5(r)=0$. By \eqref{r19}, we have $\gamma_2(r)=0$. From Lemma \ref{BETA}, it follows that $\beta_5(r)$ is non-negative.
We now consider two cases as follows:

\textbf{ Case 1:} Let $\beta_5(r)>0$.  Then it follows that $R_2(r) \geq 0$ and $T_2(r)\geq 0$. Therefore,

 \begin{eqnarray*}
\lambda_{\max}(H_{\theta}(A(5,s,a,r)))&=&\frac{1}{2\sqrt{2}}\sqrt{\beta_1(r)+\beta_2(r)\cos2\theta+\sqrt{2\beta_5(r)}\Big(\cos2\theta+\frac{\beta_4(r)}{4\beta_5(r)}\Big)}\\ &=&  \frac{1}{2\sqrt{2}}\sqrt{R_2(r)-\Big(R_2(r)-T_2(r)\Big)\sin^2\theta}.
 \end{eqnarray*}
 Hence by Theorem \ref{largest}, the numerical range of the matrix $A(5,s,a,r)$ is an elliptical disk with origin as its center.

\textbf{Case 2:} Let $\beta_5(r)=0$. Then from Remark \ref{REM},
$\beta_4(r)=0$ and $\beta_3(r)>0$. So, 
\begin{eqnarray}\label{beta_5}
\lambda_{\max}(H_{\theta}(A(5,s,a,r)))\nonumber&=&\frac{1}{2\sqrt{2}}\sqrt{\beta_1(r)+\beta_2(r)\cos2\theta+\sqrt{\beta_3(r)}}\\\nonumber &=& \frac{1}{2\sqrt{2}}\sqrt{\beta_1(r)+\beta_2(r)+\sqrt{\beta_3(r)}-2\beta_2(r) \sin^2 \theta}\\ &=& \frac{1}{2\sqrt{2}} \sqrt{R'_2(r)-(R'_2(r)-T'_2(r))\sin^2 \theta}.
\end{eqnarray}
From Lemma \ref{R'}, it follows that both $R'_2(r)$ and $T'_2(r)$ are non-negative. Hence by Theorem \ref{largest}, the numerical range of the matrix  $A(5,s,a,r)$ is an elliptical disk with centered at origin.

For both the cases, if at least one of $a,s$ and $r$ is not equal to $1$, then the matrix  $A(5,s,a,r)$ is not a normal matrix. Hence the elliptical numerical range is necessarily non-degenerate.

When $r=0$, the matrix $A(5,s,a,r)$ reduces to a weighted shift matrix. Consequently, its numerical range $W(A(5,s,a,r))$ forms a circular disk with origin as its center.\\
Conversely, if $W(A(5,s,a,r))$ is non-degenerate elliptical disk with center at origin 
 whose horizontal and vertical semi-axes have lengths $p_2(r)$ and $q_2(r)$, respectively. Then at least one of $a,s$ and $r$ is not equal to $1$, otherwise the matrix becomes Hermitian. Thus, using Theorem \ref{largest}, we have
\begin{eqnarray*}
    \lambda_{\max}(H_{\theta}(A)) &=&\sqrt{(p_2(r))^2 -((p_2(r))^2-(q_2(r))^2) \sin^2 \theta}.
\end{eqnarray*}
Comparing with \eqref{Lambda}, we have $\gamma_2(r)=0$.
 Hence by Lemma \ref{BETA}, $\beta_5(r)\geq 0$.
Consequently,
 \begin{center}
 $8(p_2(r))^2=$
         $\begin{cases}
       R_2(r)  & \mbox{ if $\beta_5(r)>0$  }\\
       R'_2(r) & \mbox{ if $\beta_5(r)=0,$}
    \end{cases}$ and
    $8(q_2(r))^2=$
    $\begin{cases}
       T_2(r)  & \mbox{ if $\beta_5(r)>0$  }\\
       T'_2(r) & \mbox{ if $\beta_5(r)=0$}.
    \end{cases}$
     \end{center} Hence $R_2(r) \geq 0$, $T_2(r) \geq 0$ or $\beta_5(r)=0$. Since $\gamma_2(r)=0$, it follows from \eqref{r19} that either $r=0$ or $r$ is a real root of  $v(t)=0$. \\
     Next, our aim is to prove that if $W(A(5,s,a,r))$ is a non-degenerate elliptical disk with origin as its center, which is not circular, then $r$ is a real root of $v(t)=0$, whereas, if  $W(A(5,s,a,r))$ forms a circular disk  with  origin as its center, then $r=0$.\\
Now, suppose $W(A(5,s,a,r))$ is a non-degenerate elliptical disk with the origin as its center, which is not circular, i.e., $p_2(r)\neq q_2(r)$. Now, if  $r= 0$ then $\beta_5(r) =0$, $p_2(r)=\sqrt{\frac{R'_2(r)}{8}}$ and $q_2(r)=\sqrt{\frac{T'_2(r)}{8}}$. By Lemma \ref{ST}, $p_2(r)=\sqrt{\frac{R'_2(r)}{8}}= \sqrt{\frac{T'_2(r)}{8}}=q_2(r)$, which is a contradiction. Therefore, $r$ is nonzero, and hence $r$ is a real root of $v(t)=0$.\\
Again, if $W(A(5,s,a,r))$ is a disk which is circular  with  origin as its center. Then, by Theorem \ref{largest}, we have $p_2(r)=q_2(r)$. If possible let $r \neq 0$,  and $r$ satisfies the equation $\gamma_2(r)=0$. Then  by Lemma \ref{BETA}, for that $r$,  $\beta_5(r) \geq 0$.
Now we consider two cases as follows:\\
\textbf{Case 1:} Let $\beta_5(r) > 0$. Then $p_2(r)=\sqrt{\frac{R_2(r)}{8}}$ and $q_2(r)=\sqrt{\frac{T_2(r)}{8}}$. 
Therefore, $R_2(r)= T_2(r)$  implies  that \begin{eqnarray*}
     \beta_2(r)=-\sqrt{2\beta_5(r)},
 \end{eqnarray*} which yields
    $16a^2r^4(s^2r^2+sr+1)=0$.
  Since $r\neq 0$, then $s^2r^2+sr+1=0$ implies $r$ is not real, this leads to a contradiction.\\
\textbf{Case 2:} Let $\beta_5(r) = 0$. Then $p_2(r)=\sqrt{\frac{R'_2(r)}{8}}$ and $q_2(r)=\sqrt{\frac{T'_2(r)}{8}}$. Since $R'_2(r) = T'_2(r)$, Lemma \ref{ST} implies that $r= 0$, which contradicts the assumption.
Hence the result follows.
\end{proof}
\begin{remark}
    In particular, by choosing  $a=s=1$ in Theorem \ref{W5}, we obtain
    \begin{eqnarray}
        v(t)\nonumber&=& 2(t^{19}-t^{18}+t^{15}-3t^{13}-t^{12}+2t^{11}+2t^9-3t^8+3t^7+t^6\\&&+3t^5-3t^4-3t^3-2t^2+2t+1)\label{Ux} \\&=&\nonumber 2(t-1)^4(t^{15}+3t^{14}+6t^{13}+10t^{12}+16t^{11}+25t^{10}+35t^9\\&&+\nonumber43t^8+48t^7+49t^6+47t^5+40t^4+29t^3+16t^2+6t+1)\\&=&\nonumber 2(t-1)^4(t^2+t+1)d_1(t),
    \end{eqnarray}
    where 
    \begin{eqnarray*}
d_1(t)&=&t^{13}+2t^{12}+3t^{11}+5t^{10}+8t^{9}+12t^{8}+15t^7+16t^6+17t^5+16t^4\\&&+14t^3+10t^2+5t+1.
\end{eqnarray*}
So, the equation $v(t)=0$ holds if and only if either  $t=1$ or $d_1(t)=0$. Since $a=s=1$, and if $t=1$, then the matrix $A(5,1,1,1)$ is Hermitian and $W(A(5,1,1,1))$ reduces to a line segment. Then it follows from Theorem \ref{W5} that $W(A(5,s,a,r))$ forms a circular disk with the origin as its center if and only if $r=0$ and a non-degenerate elliptical disk (noncircular disk) with the origin as its center if and only if $r$ is a real root of $d_1(t)=0$. The uniqueness of this root follows from  \cite[Theorem 4]{chien2011numerical}, and its numerical value is $r=-0.3884$. At this root, both $R_2(r)$ and $T_2(r)$ are positive. Hence for $a=s=1$, \cite[Theorem 4]{chien2011numerical} is a particular case of Theorem \ref{W5}.
\end{remark}
\begin{question}
    Suppose $a>0$, $s>0$ and $r$ is a real root of $v(t)=0$ as defined in \eqref{v}. Are $R_2(r)$ and $T_2(r)$ always non-negative ?
\end{question}
\section{Flat portions on the boundary of the numerical range of $A(n,s,a,-1)$}\label{section5}
From \cite[Theorem 1]{chien2011numerical}, it is known that the numerical range
 $W(A(n,r))$ of the matrix  $A(n,r)$, as  defined in \eqref{A(n,r)}, is symmetric about both the real and imaginary axes.
Using arguments parallel to those in the proof of \cite[Theorem 1]{chien2011numerical}, the following result is obtained. 
\begin{theorem} \label{symmetric}
    Let $a,s>0$,  $r \in \mathbb{R}$ and  $A=A(n,s,a,r) \in M_n$. Then $W(A)$ exhibits symmetry about both the real and imaginary axes.
\end{theorem}

Our next objective is to determine the flat portions on the boundary of the numerical range of the matrix
 $A(n,s,a,r)$ for the special case $r=-1$. Throughout this section, $\Im(e^{-i\theta}A)$ denotes imaginary part of $e^{-i\theta}A$ and $\Re(z)$ denotes the real part of a complex number $z$.

Let $A$ be an $n\times n$ tridiagonal matrix with complex entries of the following form
\begin{eqnarray}\label{tridiagonal}
A=
    \begin{bmatrix}
        a_1 & b_1 &0 &\cdots & 0\\
        c_1 & a_2 &b_2 &\cdots & 0\\
        0 & c_2 &a_3 &\cdots & 0\\
      \vdots & \vdots &\vdots & \ddots& b_{n-1}\\
      0 & 0 & \cdots &c_{n-1}& a_n\\
    \end{bmatrix}.
\end{eqnarray}

The following theorem is important for proving the main results of this section.

\begin{theorem}\cite{brown2004matrices}\label{hh}
    Let $A$ be an $n\times n$  tridiagonal matrix  of the form \eqref{tridiagonal}, which is proper. Then the boundary of the  numerical range $W(A)$ has a flat portion at an angle $\theta$ to the positive $x$-axis if and only if
    \begin{enumerate}
        \item The set $J=\{j:b_j=e^{2i\theta}\overline{c_j}\}$ is non empty;
        \item  At least two matrices among the tridiagonal Hermitian $\Im(e^{-i\theta}A_k)$ possess the same simple eigenvalue $\mu$, where $\mu$ coincide with either the minimum or maximum eigenvalue of $\Im(e^{-i\theta}A)$;\\
       (Here,  $A_k=A[j_{k-1}+1,\ldots, j_k]$; $k=1, \ldots, m=|J|+1, j_0=0, j_m=n$, and $j_1 < \cdots <j_{m-1}$ are the elements of the set $J$.) 
        \item  The  eigenvectors (unit) $x_k$ corresponding to the eigenvalue $\mu$ of $\Im(e^{-i\theta}A_k)$ are as either $\Re(e^{-i\theta}x_k^* A_
        kx_k)$ are not all the same, or there are two consecutive values of $k$ (say, $l_1$ and $l_1+1$) for which both the last coordinate of the unit eigenvector $x_{l_1}$ and first coordinate of  the unit eigenvector $x_{l_1+1}$ are non zero.\\
        (Here, $k$ ranges over the set $K\subseteq \{1,\ldots, m\}$ of indices for which $\mu$ is an eigenvalue of the matrix $\Im(e^{-i\theta}A_k)$.) 
    \end{enumerate}
\end{theorem} 
\begin{theorem}\label{2flat}
Let
     $n\geq 4$ and $a,s>0$. If $a=s$, then the boundary of the numerical range $W(A(n,s,a,-1))$  has two flat portions parallel to the $x$-axis.
\end{theorem}
\begin{proof}
We prove the theorem by considering two separate cases:

\textbf{Case 1 }($n$ is odd):
Since $n$ is an odd integer with $n\geq 5$, there exist an integer $m \geq 3$ such that $n=2m-1$. Hence  for the matrix $A(2m-1,a,a,-1)$, the set $J$, as defined in Theorem \ref{hh}(1), at $\theta =0$, is 
 $$J=\{2,4,6,\ldots,2(m-1)\}.$$ Then clearly  $|J|+1=m$.
 By Theorem \ref{hh}(2), we have
  $2=j_1<j_2=4< \cdots < j_{m-1}=2(m-1)$, together with
 $j_0= 0, j_m= n=2m-1$.

From the definition of $A_k$ in Theorem \ref{hh}(2), it follows that for the matrix $A(2m-1,a,a,-1)$ \begin{eqnarray*}
       A_1= A[1,2]=\begin{bmatrix}
           &0& &a&\\
           &-1& &0&
       \end{bmatrix},
   \end{eqnarray*}
   \begin{eqnarray*}
       A_2=A_3=\cdots=A_{m-1}=\begin{bmatrix}
           &0& &a&\\
           &-1& &0&
       \end{bmatrix}
       \,\,\,\mbox{and  $A_m=[0]$ .}
   \end{eqnarray*}
 
 Then \begin{eqnarray*}
     \Im(A_1)=\Im(A_2)=\cdots = \Im(A_{m-1})=\begin{bmatrix}
           &0& &\frac{a+1}{2i}&\\
           &\frac{-1-a}{2i}& &0&
       \end{bmatrix},
 \end{eqnarray*} and  
  clearly the eigenvalues  of $\Im(A_k)$ are $\pm \frac{a+1}{2}$ for  $k=1,2,\ldots,(m-1)$. 
Since \begin{eqnarray*}
       \Im(A)=\begin{bmatrix}
           &0& &\frac{a+1}{2i}&\\
           &-\frac{a+1}{2i}& &0&
       \end{bmatrix}\oplus \cdots \oplus \begin{bmatrix}
           &0& &\frac{a+1}{2i}&\\
           &-\frac{a+1}{2i}& &0&
       \end{bmatrix} \oplus [0],
   \end{eqnarray*}
    the eigenvalues of $\Im(A)$ are $0$ and $\pm \frac{a+1}{2}$.  By selecting $\mu$ (as defined in Theorem \ref{hh}) to be either $\frac{a+1}{2}$ or $-\frac{a+1}{2}$, the conditions of Theorem \ref{hh}(2) are satisfied.

 Since  $\Im(A_1)=\Im(A_2)=\cdots=\Im(A_{m-1})$, we choose $l=1$ and $\mu=\frac{a+1}{2}$ in Theorem \ref{hh}(3). It follows that  $x_1=x_2=\frac{1}{\sqrt{2}}(1,i)$, which satisfies the conditions of Theorem \ref{hh}(3).
Hence, the numerical range $W(A(2m-1,a,a,-1))$ contains a flat portion on its boundary parallel to the $x$-axis. 
Moreover, by Theorem \ref{symmetric}, $W(A(2m-1,a,a,-1))$ is symmetric about the real axis. 
Since  $A(2m-1,a,a,-1)$ is not  Hermitian, the desired result follows.

\textbf{Case 2} ($n$ is even):
By a similar argument, the same conclusion holds for even values of $n$.
\end{proof}
\begin{remark}\label{rr} 
 By Theorem \ref{2flat}, for $n\geq 4$, the boundary of the numerical range $W(A(n,a,a,-1))$ contains two flat portions parallel to the $x$-axis. Since the largest and smallest eigenvalues of $\Im(A)$ are $\frac{a+1}{2}$ and $-\frac{a+1}{2}$ respectively, it follows from the convexity of  $W(A(n,a,a,-1))$ that the flat portions parallel to $x$- axis lies on the lines     $$y=\pm \frac{a+1}{2}.$$
 
        In particular, if we take $a=s=1$ and $n\geq 5$ is odd in Theorem \ref{2flat}, it follows that numerical range of the matrix $A(n,-1)$ (see \eqref{A(n,r)}) has two flat portions parallel to $x$-axis, and lying on the lines $y=\pm1$, which is also reflected in \cite[Theorem 7]{chien2011numerical}.
    
\end{remark}
\begin{theorem}\label{flat}
Let $n\geq 4$ and $a,s>0$. If $a=s\neq1$, the boundary of the numerical range $W(A(n,s,a,-1))$ contains no flat portions parallel to the $y$-axis.
\end{theorem}
\begin{proof}
     For $\theta =\frac{\pi}{2}$, according to condition of Theorem \ref{hh}(1), the relation $b_j=e^{i\pi}\overline{c_j}$ implies either $a=1$ or $a=0$, which is not possible. Thus, according to Theorem \ref{hh}(1), the  set $J$ is empty. Hence the result follows.
\end{proof}
The following is a supporting example of the above two results.
\begin{example}
    Consider $a=s=2$ and the order of the matrix is $n=6$. Then, the shape of  $W(A(6,2,2,-1))$ is given in the Figure \ref{W(A(6,2,2,-1))}.
\end{example}
\begin{figure}[ht!]
         \centering
         \includegraphics[scale=.28]{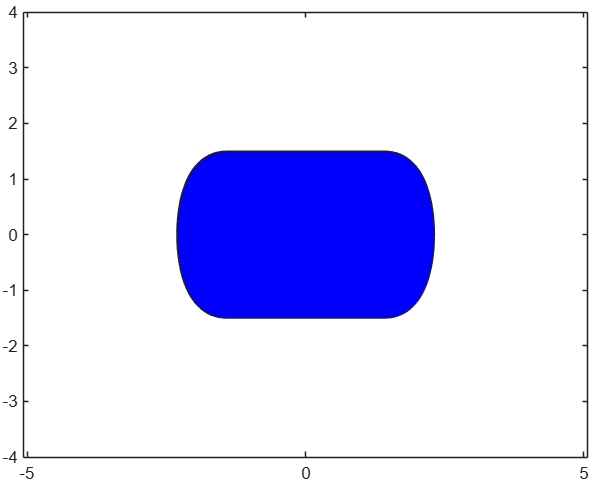}
        \caption{ $W(A(6,2,2,-1))$}\label{W(A(6,2,2,-1))}
\end{figure}

\begin{theorem}\label{parallel_y}
    Let $n\geq 5$ and $s>0$. If $a=1$, then the boundary of the numerical range $W(A(n,s,a,-1))$ exhibits two flat portions parallel to the $y$-axis.
\end{theorem}
\begin{proof}
We establish the theorem by analyzing two distinct cases:

\textbf{Case 1} ($n$ is odd): Since $n$ is an odd integer with $n\geq 5$, there exist an integer $m \geq 3$ such that $n=2m-1$. Then for the matrix $A(2m-1,s,1,-1)$, the set $J$, as defined in Theorem \ref{hh}(1), at $\theta=\frac{\pi}{2}$, is $$J=\{1,3,\ldots,2m-3\}.$$ Then clearly  $|J|+1=m$. By Theorem \ref{hh}(2), we have $1=j_1<j_2=3< \cdots < j_{m-1}=2m-3$, with $j_0= 0, j_m= n=2m-1$.

From the definition of $A_k$ in Theorem \ref{hh}(2), it follows that for the matrix $A(2m-1,s,1,-1)$
\begin{eqnarray*}
    A_1=\begin{bmatrix}
        0
    \end{bmatrix},
    \,\, \mbox{and}\,\,\,
       A_2=A_3=\cdots=A_{m}=\begin{bmatrix}
           &0& &1&\\
           &s& &0&
       \end{bmatrix}.
   \end{eqnarray*}
   Then for $k=2,3,\ldots, m$, we have
   \begin{eqnarray*}
       \Im(e^{-i\frac{\pi}{2}}A_k)=-\frac{1}{2}(A_k+A_k ^*)=\begin{bmatrix}
           &0& &-\frac{1+s}{2}&\\
           &-\frac{1+s}{2}& &0&
       \end{bmatrix},
   \end{eqnarray*}
  and  
  clearly the eigenvalues  of $ \Im(e^{-i\frac{\pi}{2}}A_k)$ are $\pm \frac{1+s}{2}$ for $k=2,3,\ldots, m$.

   Again since 
   \begin{eqnarray*}
       \Im(e^{-i\frac{\pi}{2}}A)&=&-\frac{1}{2}(A+A^*)\\&=&[0]\oplus \begin{bmatrix}
           &0& &-\frac{1+s}{2}&\\
           &-\frac{1+s}{2}& &0&
       \end{bmatrix}\oplus \cdots \oplus \begin{bmatrix}
           &0& &-\frac{1+s}{2}&\\
           &-\frac{1+s}{2}& &0&
       \end{bmatrix},
   \end{eqnarray*} the eigenvalues of $\Im(e^{-i\frac{\pi}{2}}A)$ are $0$, $\pm\frac{1+s}{2}$. By selecting $\mu$ (as defined in Theorem \ref{hh}) to be either $\frac{1+s}{2}$ or $-\frac{1+s}{2}$, the conditions of Theorem \ref{hh}(2) are satisfied.

Since  $\Im(A_2)=\cdots=\Im(A_m)$, we choose $l=2$ and $\mu=\frac{s+1}{2}$ in Theorem \ref{hh}(3). It follows that $x_2=x_3=\frac{1}{\sqrt{2}}(1,-1)$  which satisfies the conditions of Theorem \ref{hh}(3).
Hence, the numerical range $W(A(2m-1,s,1,-1))$ contains a flat portion on its boundary parallel to the $y$-axis. 
Moreover, by Theorem \ref{symmetric},  $W(A(2m-1,s,1,-1))$ is symmetric about the imaginary axis. 
Since  $A(2m-1,s,1,-1)$ is not  Hermitian, the desired result follows.

 \textbf{Case 2} ($n$ is even):  
By an analogous argument, the same conclusion follows for even values of  $n$.
\end{proof}
The following is a supporting example of the above result.
\begin{example}
    Consider $a=1, s=5$ and the order of the matrix is $n=7$. Then, the shape of  $W(A(7,5,1,-1))$ is given in the Figure \ref{W(A(7,5,1,-1))}.
\end{example}
\begin{figure}[ht!]
         \centering
         \includegraphics[scale=.26]{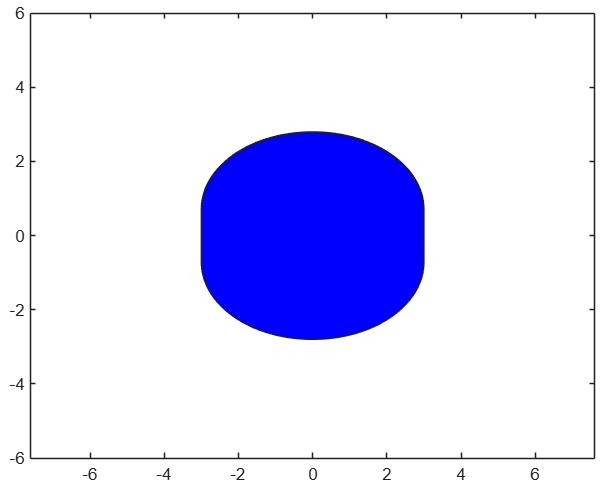}
        \caption{$W(A(7,5,1,-1))$}\label{W(A(7,5,1,-1))}
\end{figure}

\begin{remark}
By Theorem \ref{parallel_y}, for $n\geq 5$, the boundary of   $W(A(n,s,1,-1))$ contains two flat portions parallel to the $y$-axis. Since the largest and smallest eigenvalues of $\Re(A)$ are $\frac{s+1}{2}$ and $-\frac{s+1}{2}$ respectively, it follows from the convexity of $W(A(n,s,1,-1))$ that the flat portions parallel to the $y$-axis lies on the lines     $$x=\pm \frac{s+1}{2}.$$ 
 
        In particular, if we take $s=1$  and $n\geq 5$ is odd in Theorem \ref{parallel_y},  it follows that  $W(A(n,-1))$ has two flat portions parallel to the $y$-axis, and lying on the lines $x=\pm1$, which is also reflected in \cite[Theorem 7]{chien2011numerical}. 
\end{remark}
\begin{theorem}\label{fflat}
    Let $n\geq 4$ and  $a,s>0$. If $a\neq 1$ and $a\neq s$, then the boundary of the numerical range $W(A(n,s,a,-1))$ has no flat portions in any direction. 
\end{theorem}
\begin{proof}
     According to condition of Theorem \ref{hh}(1), the relation  $b_j=e^{2i\theta}\overline{c_j}$ yields $a=-e^{2i\theta}$ or  $a=se^{2i\theta}$, implies either  $a=\pm1$ or $a=\pm s$, which is not possible for any $\theta$ by the given conditions. Thus, according to Theorem \ref{hh}(1), the  set $J$ is empty. Hence the result follows. 
\end{proof}
The following is a supporting example of the above result.
\begin{example}
    Consider $a=2, s=7$ and the order of the matrix is $n=8$. Then, the shape of $W(A(8,7,2,-1))$ is given in the Figure \ref{W(A(8,7,2,-1))}.
\end{example}
\begin{figure}[ht!]
         \centering
         \includegraphics[scale=.4]{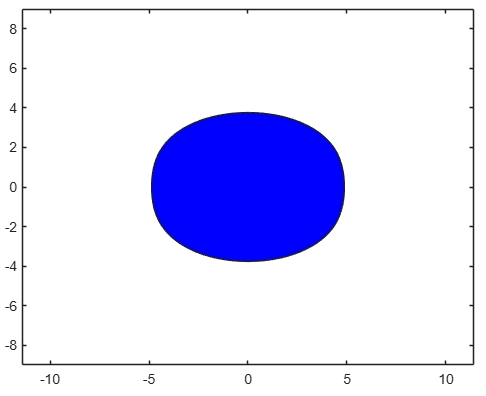}
        \caption{$W(A(8,7,2,-1))$}\label{W(A(8,7,2,-1))}
\end{figure}

\begin{remark}
Combining Theorems \ref{2flat} and \ref{parallel_y}, for $n\geq 5$, we conclude that for $a=s=1$, the boundary of $W(A(n,-1))$ contains two flat portions parallel to the $x$-axis, lying on the lines $y=\pm 1$, and two flat portions parallel to the $y$-axis, lying on the lines $x=\pm 1$. Moreover, Theorem \ref{hh}(1) implies that  the  boundary of $W(A(n,-1))$ possess only that flat portions which are parallel to the coordinate axes. Hence for $n\geq 5$, the boundary of $W(A(n,-1))$ contains exactly four flat portions, which also follows from \cite[Theorem 8]{chien2011numerical}.
\end{remark}
\begin{note}
 All figures are plotted using the program given in \cite{cowen1995effective}.
\end{note}
\section{Declarations}
\textit{Acknowledgements:} Mr. Arobinda Ghosh would like to thank UGC, Govt. of India for the financial support (NTA Ref. No. 211610122278) in the form of fellowship.\\
\textit{Author Contributions:} All the authors contributed equally to this manuscript and reviewed it. \\
 \textit{Data Availability :} No datasets were generated or analyzed during the current study. \\
\textit{Conflict of interest:} There is no competing interest.\\

\bibliographystyle{amsplain}

\begin{thebibliography}{99}


\bibitem{brown2004matrices}Brown, E.S. and Spitkovsky, I.M.: On flat portions on the boundary of the numerical range. Linear Algebra Appl., 390, 75-109, (2004).



\bibitem{chien1996numerical}Chien, M.T.: On the numerical range of tridiagonal operators. Linear Algebra Appl., 246, 203-214, (1996).




\bibitem{chien2011numerical} Chien, M.T. and Nakazato, H.: The numerical range of a tridiagonal operator. J. Math. Anal. Appl., 373(1), 297-304, (2011).

\bibitem{chien1998geometric} Chien, M.T., Yeh, L., Yeh, Y.T. and  Lin, F.Z.: On geometric properties of the numerical range. Linear Algebra Appl., 274(1-3), 389-410, (1998).





\bibitem{chien2015numerical}Chien, R.T. and Spitkovsky, I.M.: On the numerical ranges of some tridiagonal matrices. Linear Algebra Appl., 470, 228-240, (2015).






\bibitem{cowen1995effective}Cowen, C.C. and Harel, E.: An effective algorithm for computing the numerical range. Unpublished manuscript, (1995).


\bibitem{gustafson1997numerical} Gustafson, K.E. and Rao, D.K.:  Numerical range. The Field of Values of Linear Operators and Matrices (1-26). Springer, New York, (1997).


\end{thebibliography}

\end{document}